\documentclass[12pt]{amsart}
\usepackage[margin=1in]{geometry}
\usepackage{ifdraft}

\usepackage{amsmath,amsthm,amssymb, amsfonts}
\usepackage{enumitem}
\usepackage{hyperref}
\usepackage{graphicx}
\usepackage{mathtools,mathrsfs, bbm}
\usepackage{stmaryrd}
\usepackage{tikz-cd}

\allowdisplaybreaks

\newcommand{\A}{{\mathcal{A}}}
\newcommand{\affine}{{\mathbb{A}}}
\newcommand{\C}{{\mathbb{C}}}
\newcommand{\g}{{\mathfrak{g}}}
\newcommand{\gl}{{\mathfrak{gl}}}

\renewcommand{\k}{{\mathbbm{k}}}

\renewcommand{\L}{{\mathcal{L}}}

\newcommand{\N}{{\mathbb{N}}}

\newcommand{\T}{{\mathbb{T}}}
\newcommand{\U}{{\mathcal{U}}}
\newcommand{\V}{{\mathcal{V}}}
\newcommand{\Z}{{\mathbb{Z}}}

\newcommand{\Coind}{{\operatorname{Coind}}}
\newcommand{\Der}{{\operatorname{Der}}}

\newcommand{\Hom}{{\operatorname{Hom}}}

\newcommand{\End}{{\operatorname{End}}}

\newcommand{\DR}{{\textup{dR}}}

\newcommand{\GF}{{\textup{GF}}}
\newcommand{\fin}{{\textup{fin}}}
\newcommand{\vphi}{{\varphi}}
\newcommand{\hphi}{{\tilde{\varphi}}}

\renewcommand{\sp}{{\#}}
\newcommand{\wo}{\widehat{\otimes}}

\newcommand{\dx}[1]{{\frac{\partial}{\partial x_{#1}}}}

\newtheorem{theorem}{Theorem}[section]
\newtheorem{lemma}[theorem]{Lemma}
\newtheorem{proposition}[theorem]{Proposition}
\newtheorem{corollary}[theorem]{Corollary}

\theoremstyle{definition}
\newtheorem{definition}{Definition}[section]

\newtheorem{remark}[definition]{Remark}

\title{Gelfand-Fuks cohomology of vector fields on algebraic varieties}
\author{Yuly Billig}
\address{School of Mathematics and Statistics, Carleton University, Ottawa, Canada}
\author{Kathlyn Dykes}

\begin{document}
\let\thefootnote\relax\footnotetext{{\it 2020 Mathematics Subject Classification.}
Primary 17B56; Secondary 17B66}

\begin{abstract}
For an affine algebraic variety, we introduce algebraic Gelfand-Fuks cohomology of polynomial vector fields with coefficients in differentiable $AV$-modules. Its complex is given by cochains that are differential operators in the sense of Grothendieck. Using the jets of vector fields, we compute this cohomology for varieties with uniformizing parameters. We prove that in this case, Gelfand-Fuks cohomology with coefficients in a tensor module decomposes as a tensor product of the de Rham cohomology of the variety and the cohomology of the Lie algebra of vector fields on affine space, vanishing at the origin. We explicitly compute this cohomology for affine space, the torus, and Krichever-Novikov algebras. 
\end{abstract}

\maketitle

\section{Introduction}

The study of Lie algebras of vector fields goes back to the foundational works of Sophus Lie \cite{Lie1880} and \'Elie Cartan \cite{Cartan1909}.
Starting in the late 1960s, Gelfand's school actively studied cohomology theory for the Lie algebras of vector fields. It was realized, that beyond the simplest examples, computation of the general cohomology of these algebras is intractable. To make computations possible, a special cohomology theory was introduced by Gelfand and Fuks \cite{GelfandFuks}. Gelfand-Fuks cohomology was defined in an analytic setting of $\mathcal{C}^\infty$ varieties.  
For a comprehensive overview of this work, see \cite{Fuks1986}. In the context of complex varieties, a subcomplex, called the diagonal complex, was introduced, see Tsujishita \cite{Tsujishita1981} for many of these results. 

Although there was much progress in this theory, the standard approach was computationally difficult, causing the area to stall for many years. There has been some recent progress. On the algebraic side, Hennion and Kapranov \cite{HennionKarpranov2022} solved an outstanding open problem, proving that  the Lie algebra cohomology of vector fields on a smooth affine variety with coefficients in the trivial module is finite dimensional. To do this, they developed an approach using factorization algebras. On the topological side, the work of Williams \cite{Williams24} computed the local cohomology of vector fields on a manifold, connecting the diagonal complex with the work of Hennion and Kapranov.

In this paper, we introduce algebraic Gelfand-Fuks cohomology, working in a purely algebraic setting. We take advantage of the new machinery that was recently developed in representation theory. A key tool that we make use of in this paper is $AV$-modules. $AV$-modules are modules that admit simultaneous actions of a commutative algebra of polynomial functions on an affine variety and the Lie algebra of vector fields, such that these actions are compatible via the Leibniz rule. $AV$-modules were instrumental for the classification of simple weight modules for the Lie algebra of vector fields on a torus \cite{BilligFutorny2016} and on an affine space \cite{XL2023, GS2022}. In particular, we will need the notion of a differentiable $AV$-module, i.e., a module for which the representation is a differential operator in the sense of Grothendieck. Differentiable modules are studied in \cite{BouazizRocha}. 

The main goal of this paper is to develop a new algebraic approach to Gelfand-Fuks cohomology theory. The results we obtain are not surprising -- they agree with those obtained in the $\mathcal{C}^\infty$ setting. 

\begin{theorem}\label{theorem:intro}
Let $X$ be an affine algebraic variety with uniformizing parameters $x_1, \dots, x_n$. Let $V$ be the Lie algebra of vector fields and $A$ be the algebra of regular functions on $X$. Let $\L_+$ be the Lie algebra of vector fields on an affine space $\mathbb{A}^n$ vanishing at the origin. 
Let $W$ be a finite-dimensional module for $\L_+$, and let $A \otimes W$
be the corresponding tensor module on $X$. 
Then we have the following isomorphism:
\[
H^*_{\GF}(V, A \otimes W) \cong H_{\DR}^*(X) \otimes H^*(\L_+, W).
\] 
\end{theorem}

Let us outline our method for computing the algebraic Gelfand-Fuks cohomology.
First, we show that the cohomology of $V$ with values in a differentiable $AV$-module $M$ coincides with the $A$-linear cohomology of the Lie algebra of polynomial jets $A \# V$:
\[
H^* (V, M) \cong H^*_A (A\# V, M).
\]
If we apply this isomorphism to the Gelfand-Fuks complex, on the right-hand side we will get the finite cohomology of the completed Lie algebra $A\widehat{\#}V$ of $\infty$-jets of vector fields on $X$:
\[
H^*_\text{GF} (V, M) \cong H^*_{A,\text{fin}} (A\widehat{\#} V, M).
\]
By a result of \cite{BilligIngallsNasr, BilligIngalls}, on a variety with uniformizing parameters, the Lie algebra $A\widehat{\#}V$ of $\infty$-jets of vector fields is isomorphic to a semi-direct product:
\[
A\widehat{\#} V \cong V \ltimes (A \wo \L_+).
\]
To compute $H^*_{A,\text{fin}} ( V \ltimes (A \wo \L_+), A \otimes W)$, we establish a K\"unneth formula for this semi-direct product (Theorem 5.2), and obtain
\[
H^*_{A,\text{fin}} ( V \ltimes (A \wo \L_+), A \otimes W) \cong
H^*_A (V, A) \otimes H^*_\text{fin} (\widehat{\L}_+, W).
\]
Finally, we simplify the right-hand side by showing both $H^*_\text{fin} (\widehat{\L}_+, W) = H^* (\L_+, W)$, and that $H^*_A (V, A)$ 
is the de Rham cohomology of $X$:
\[
H^*_A (V, A) \cong H_{\DR}^*(X).
\]

The paper is organized as follows. In Section \ref{section:GFcohomology}, we introduce the main notion of the paper: algebraic Gelfand-Fuks cohomology of vector fields with values in a differentiable $AV$-module. We review some necessary background on the jets of vector fields in \'etale charts in Section \ref{section:background}, which will be our main tools. We describe the de Rham cohomology as the Lie algebra cohomology of vector fields in Section \ref{section:deRham}. We establish a K\"unneth formula for semi-direct products of Lie algebras of vector fields in Section \ref{section:Kunneth}. In Section \ref{section:etalecharts}, we compute the algebraic Gelfand-Fuks cohomology with coefficients in tensor modules for varieties with uniformizing parameters and we explicitly compute this cohomology for the affine space, the torus, and Krichever-Novikov algebras in the last Section of the paper. 

\

{\bf Acknowledgements:}
Y.B. gratefully acknowledges support with a Discovery grant from the Natural Sciences and Engineering Research Council of Canada. 

\section{Algebraic Gelfand-Fuks cohomology}\label{section:GFcohomology}

Let $(A, V)$ be a Rinehart pair, that is, let $A$ be a unital commutative associative algebra and $V$ be a Lie algebra, such that $V$ is an $A$-module, 
and $A$ is a $V$-module with $V$ acting on $A$ by derivations, satisfying
\begin{align}\label{equation:LieRinehartpair}
[\eta, f \mu] &= \eta(f) \mu + f [\eta, \mu], \\
(f \mu) (g) &= f (\mu(g))
\ \ \ \text{for} \, f, g \in A, \ \eta, \mu \in V.
\end{align}

A prime example of a Rinehart pair is given by vector fields on a variety. 
Let $X$ be a smooth irreducible affine algebraic variety over a field $\k$ of characteristic zero. To construct a Rinehart pair, we take $A$ to be the coordinate ring of $X$ and $V$ the Lie algebra of polynomial vector fields on $X$, $V = \Der\,A$. 

\begin{definition}\label{definition:AVmodules} 
For a Rinehart pair $(A, V)$, a module $M$ is called an \emph{$AV$-module} if $M$ is an $A$-module, a $V$-module, and their actions are compatible via the Leibniz rule
\begin{align}\label{equation:Leibniz}
\eta (f m) = \eta(f) m + f( \eta \,m)
\ \  \ \text{for} \, \eta \in V, f \in A, m \in M.
\end{align}
 
\end{definition}

In the context of algebraic varieties, $AV$-modules are generalizations of $D$-modules. Every $D$-module is automatically an $AV$-module, however there is an abundance of natural $AV$-modules on $X$ which do not admit a $D$-module structure. In particular, vector fields $V$ themselves are an $AV$-module, but not a $D$-module. In the category of $AV$-modules, $D$-modules are characterized by an additional axiom:
\[
(f \eta)\, m = f\, (\eta \,m).
\]

\begin{definition}\label{definition:differentiable}
An $AV$-module $M$ is called \emph{differentiable} (or $N$-differentiable) if there exists $N \in \N$ such that
\begin{equation}
\sum_{j=0}^N (-1)^j \binom{N}{j} f^{N-j} \left( f^j \eta \right) m = 0
\ \ \ \text{for all} \ f\in A, \eta \in V, m \in M.
\end{equation}
\end{definition}
Being a differentiable $AV$-module is equivalent to the condition that the representation map \ $V \rightarrow \End(M)$ is a differential operator in the sense of Grothendieck \cite{Grothendieck}. 
$D$-modules are precisely 1-differentiable $AV$-modules.
By \cite[Proposition 4.7]{BouazizRocha}, every $AV$-module which is finitely generated as an $A$-module is differentiable. 

Consider the cochain complex $\{ C^k(V, M) = \Hom (\bigwedge^k V, M), d\}$
with the Chevalley-Eilenberg differentials:
\begin{align*}
(d_k\vphi)(\eta_1, \dots, \eta_{k+1}) &= \sum_{1 \leq s< t \leq k+1} (-1)^{s+t-1} \vphi( [\eta_s, \eta_t], \eta_1, \dots, \hat{\eta}_s, \dots, \hat{\eta}_t, \dots, \eta_{k+1}) \\
& \quad + \sum_{1 \leq s \leq k+1} (-1)^s \eta_s \, \vphi(\eta_1, \dots, \hat{\eta}_s, \dots, \eta_{k+1})
\end{align*}
for $\vphi \in C^k(V, M)$, $\eta_1, \dots, \eta_{k+1} \in V$ and $k \in \Z_+$, where the notation $\hat{\eta}_i$ means that this element is omitted.

Next, we define the main notion of this paper -- \emph{algebraic Gelfand-Fuks} cohomology with values in a differentiable $AV$-module $M$. 
 For each $k \in \N$, define
\begin{align*}
C_{\GF}^k(V, M) = \big\{ \vphi \in C^k(V ,M) \ \big| \ & \exists p > 0 \  \ \forall f \in A, \, \forall \eta_1, \eta_2, \dots, \eta_k \in V, \\
&  \sum_{i=0}^p (-1)^i \binom{p}{i} f^{p-i} \vphi(f^i \eta_1, \eta_2, \dots, \eta_k)=0 \big\}.
\end{align*}
For $k=0$ we set $C_{\GF}^0 (V, M) = M$.
The differential maps $d_k: C_\GF^k(V, M) \rightarrow C_\GF^{k+1}(V,M)$ are the
usual  Chevalley-Eilenberg differentials. We first show that this indeed forms a subcomplex so that we can take its cohomology. 

\begin{lemma}
Let $M$ be a differentiable $AV$-module. Then
$\{C^*_{\GF}(V, M), d\}$ is a subcomplex of the Chevalley-Eilenberg complex
$\{ C^*(V, M), d\}$.
\end{lemma}

\begin{proof}
Let $\vphi \in C_{\GF}^k(V, M)$ and let $p\in \N$ such that $\sum_{i=0}^p (-1)^i \binom{p}{i} f^{p-i} \vphi(f^i \eta_1, \dots, \eta_k)=0$. First note that if such a $p$ exists for $\vphi$, then this sum is also zero for every $q \geq p$:
\begin{align*}
\sum_{i=0}^{p+1}(-1)^{i} \binom{p+1}{i} f^{p+1-i} \vphi(f^i \eta_1, \dots, \eta_k) & = f \sum_{i=0}^{p}(-1)^{i} \binom{p}{i} f^{p-i} \vphi(f^i \eta_1, \dots, \eta_k) \\ & \quad - \sum_{i=0}^{p}(-1)^{i} \binom{p}{i} f^{p-i} \vphi(f^i (f\eta_1), \dots, \eta_k).
\end{align*}
Thus we can assume without a loss of generality that $M$ is $(p+1)$-differentiable.
Consider arbitrary $f \in A, \ \eta_1, \dots, \eta_{k+1} \in V$. The differential behaves as follows:
\begin{align}\label{equation:equation1}
\begin{split}
(d\vphi)(f^i \eta_1, \eta_2, \dots, \eta_{k+1}) 
&= \sum_{2 \leq s < t \leq k+1 }(-1)^{s+t} \,\vphi (f^i \eta_1,[\eta_s, \eta_t], \eta_2, \dots,\hat{\eta}_s, \dots, \hat{\eta}_t, \dots, \eta_{k+1}) \\
& \quad + \sum_{2 \leq s \leq k+1} (-1)^s \,\vphi (f^i [\eta_1,\eta_s], \eta_2, \dots, \hat{\eta}_s, \dots, \eta_{k+1}) \\
& \quad + \sum_{2 \leq s \leq k+1 }(-1)^{s} \,\eta_s \,\vphi(f^i\eta_1, \dots, \hat{\eta}_s, \dots, \eta_{k+1}) \\
& \quad - \sum_{2 \leq s \leq k+1 }(-1)^{s}  \,\vphi(\eta_s(f^i) \eta_1, \eta_2, \dots, \hat{\eta}_s, \dots, \eta_{k+1}) \\
&\quad - (f^i\eta_1) \vphi(\eta_2, \dots,  \eta_{k+1}). 
\end{split}
\end{align}
When we consider the summation $\sum_{i=0}^{p+1} (-1)^i \binom{p+1}{i} f^{p+1-i} (d\vphi)(f^i \eta_1, \eta_2, \dots, \eta_{k+1})$, the first two sums in (\ref{equation:equation1}) will be zero as $\vphi$ is a Gelfand-Fuks cochain. The last sum is also zero since $M$ is $(p+1)$-differentiable.
Let us show that for each of the remaining terms the summation will result in zero for each $s$. For the third term we get:
\begin{align*}
\sum_{i=0}^{p+1} (-1)^i \binom{p+1}{i}& f^{p+1-i} 
\eta_s \,\vphi(f^i\eta_1, \dots, \hat{\eta}_s, \dots, \eta_{k+1}) \\
&= \eta_s \left( \sum_{i=0}^{p+1} (-1)^i \binom{p+1}{i} f^{p+1-i} 
\vphi(f^i\eta_1, \dots, \hat{\eta}_s, \dots, \eta_{k+1}) \right) \\
&- \sum_{i=0}^{p+1} (-1)^i \binom{p+1}{i} \eta_s(f^{p+1-i}) 
 \,\vphi(f^i\eta_1, \dots, \hat{\eta}_s, \dots, \eta_{k+1}).
\end{align*}
The first part is zero since $\vphi$ is a Gelfand-Fuks cochain. For the second part:
\begin{align*}
- \eta_s(f) \sum_{i=0}^{p+1} (-1)^i \binom{p+1}{i} (p+1-i) f^{p-i} 
 \,\vphi(f^i\eta_1, \dots, \hat{\eta}_s, \dots, \eta_{k+1}) \\
= - (p+1) \eta_s(f) \sum_{i=0}^{p} (-1)^i \binom{p}{i} f^{p-i} 
 \,\vphi(f^i\eta_1, \dots, \hat{\eta}_s, \dots, \eta_{k+1}),
\end{align*}
which also vanishes. Now for the fourth term in (\ref{equation:equation1}):
\begin{align*}
&\sum_{i=0}^{p+1} (-1)^i \binom{p+1}{i} f^{p+1-i} 
\vphi(\eta_s(f^i) \eta_1, \eta_2, \dots, \hat{\eta}_s, \dots, \eta_{k+1}) \\
&= \sum_{i=0}^{p+1} (-1)^i i \binom{p+1}{i} f^{p+1-i} 
\vphi(f^{i-1} \eta_s(f) \eta_1, \eta_2, \dots, \hat{\eta}_s, \dots, \eta_{k+1}) \\
&= (p+1) \sum_{j=0}^{p} (-1)^{j+1} \binom{p}{j} f^{p-j} 
\vphi(f^{j} \eta_s(f) \eta_1, \eta_2, \dots, \hat{\eta}_s, \dots, \eta_{k+1}),
\end{align*}
which is again zero by the Gelfand-Fuks condition.

Thus $d\vphi \in C^{k+1}_{\GF}(V, M)$ so $\{C_{\GF}^k(V, M), d\}$ is itself a complex. 
\end{proof}

\begin{definition}\label{definition:GelfandFukscohomology}
The cohomology of the subcomplex $\{C^k_{\GF}(V, M)\}$ is called the \emph{algebraic Gelfand-Fuks cohomology} of $V$ with coefficients in $M$, denoted by $H_{\GF}^*(V, M)$.  
\end{definition}

Conceptually, Gelfand-Fuks cochains are differential operators
$\bigwedge^k V \rightarrow M$ (in the sense of Grothendieck \cite{Grothendieck}). The Gelfand-Fuks complex introduced above is an algebraic analogue of the diagonal complex defined by Gelfand-Fuks \cite{GelfandFuks} in the context of $\mathcal{C}^\infty$ manifolds. In the diagonal complex the value of a cochain at a point depends only on the values at that point of a finite number of partial derivatives of the arguments. We shall make use of the jets of vector fields to see that algebraic Gelfand-Fuks cochains possess the same property.

\section{Rinehart enveloping algebras and jets of vector fields}\label{section:background}

In this section, we review the results on jets of vector fields in \'etale charts from  \cite{BilligIngalls}. For a Lie algebra $V$, we denote by $\U(V)$ the universal enveloping algebra. First, recall the definition of the enveloping algebra of a Rinehart pair.

\begin{definition}
For a Rinehart pair $(A, V)$, the \emph{weak Rinehart enveloping algebra} is the smash product of $A$ with $\U(V)$, denoted $A \sp \U(V)$. 
\end{definition}

Here, the smash product is the vector space $A \otimes_{\k} \U(V)$, with defining relations
\[
f \sp \eta \cdot g \sp \mu = f \eta(g) \sp \mu + fg \sp \eta \mu
\ \ \ \ \text{for} \  f,g \in A, \eta, \mu \in V.
\]

\begin{definition}
For a Rinehart pair $(A, V)$, the \emph{strong Rinehart enveloping algebra} $D(A,V)$ is the quotient of $A \sp \U(V)$ by the ideal $\left< f \sp \eta - 1 \sp (f\eta) \, | \, f\in A, \eta \in V \right>$. 
\end{definition}

In the context of vector fields on a smooth algebraic variety, $D(A,V)$ is the algebra $D$ of differential operators. 

We call an $AV$-module $M$ \emph{strong}, if it is a module for $D(A,V)$. Equivalently, $M$ is a 1-differentiable $AV$-module. These are $D$-modules in the context of algebraic varieties.

It is an important fact that the subspace $A \otimes_{\k} V \subset A \sp \U(V)$, denoted by $A \sp V$, is a Lie subalgebra with Lie bracket:
\begin{align*}
[f \sp \eta, g \sp \mu] &=  f \eta(g) \sp \mu - g \mu(f) \sp \eta  + fg \sp [\eta, \mu].
\end{align*}

In fact, $(A, A \sp V)$ is again a Rinehart pair, and $D(A, A\sp V) \cong A \sp U(V)$ \cite[Lemmas 6-7]{BilligIngalls}.
This implies that every $AV$-module $M$ is a strong module over the pair
$(A, A \sp V)$.

In the context of vector fields on an affine variety $X$, we will call $A \sp V$ the \emph{Lie algebra of polynomial jets of vector fields on $X$}.
A completion of this Lie algebra will yield the $\infty$-jets of vector fields.

To construct this completion, 
consider the kernel $\Delta \subset A \otimes_\k A$ of the multiplication map m$: A \otimes_\k A \rightarrow A$. The completion of the tensor product $A \otimes_\k A$ (see \cite{AtiyahM} for details on completions), called the \emph{algebra of jets of functions}, is defined as
\[
A \widehat{\otimes} A  = \varprojlim\limits_s (A \otimes_\k A) /\Delta^s.
\]

As $\Delta^k \otimes_A V$ are ideals of Lie algebra $A \sp V$, we can define the completion of $A \sp V$. 

\begin{definition}\label{definition:algebrajets}
The \emph{Lie algebra of jets of vector fields} is 
\[
A \widehat{\sp} V = \varprojlim\limits_s (A \sp V)\,/\,(\Delta^s \otimes_A V).
\] 
\end{definition}

Since $\mathop\cap\limits_s \Delta^s \otimes_A V = (0)$, $A \sp V$ is a Lie subalgebra in $A \widehat{\sp} V$.

By Lemma 12(3) from \cite{BilligIngalls}, the ideal $\Delta^s$ is spanned by the elements
\[
(g \otimes 1) (1 \otimes f_1 - f_1 \otimes 1) \cdot \ldots \cdot (1 \otimes f_s - f_s \otimes 1)
\]
with $g, f_1, \ldots, f_s \in A$. A linearization argument shows that this ideal is also spanned by the elements of the form
\[
(g \otimes 1) (1 \otimes f - f \otimes 1)^s.
\]
We conclude that $\Delta^s \otimes_A V$ is spanned by the elements 
\[
\sum_{i=0}^s (-1)^i \binom{s}{i} g f^{s-i} \# f^i \eta,
\]
with $f, g \in A, \eta \in V$.
Thus an $AV$-module $M$ is $s$-differentiable if and only if it is annihilated by $\Delta^s \otimes_A V$. This implies that differentiable $AV$-modules are also modules for the Lie algebra of jets of vector fields $A \widehat{\sp} V$.

Note that we have Lie subalgebra $1 \sp V \cong V$ in $A \sp V$. 
We will now show that the cohomology of $V$ is the same as the $A$-linear cohomology of $A \sp V$. 


For a Rinehart pair $(A, V)$ and a strong $AV$-module $M$,
consider the complex $\{C_A^k(V, M) = \Hom_A(\bigwedge^k_A V, M), d\}$ of $A$-linear cochains with the usual differential maps. We call the cohomology of this complex the \emph{$A$-linear cohomology of $V$ with coefficients in $M$}. The following Lemma shows that this complex is well-defined.
\begin{lemma}
Let $M$ be a strong $AV$-module, and let $\vphi \in C_A^k(V, M)$. Then 
$d\vphi \in C_A^{k+1}(V, M)$.
\end{lemma}
\begin{proof}
\begin{align*}
&d\vphi(f_1 \eta_1, \ldots, f_{k+1} \eta_{k+1}) \\
&= \sum_{1 \leq t \leq k+1} (-1)^t f_t \eta_t \, \vphi(f_1 \eta_1, \ldots, \widehat{f_t \eta_t}, \ldots, f_{k+1} \eta_{k+1}) \\
& \quad + \sum_{1 \leq s < t \leq k+1} (-1)^{s+t-1} \vphi( [f_s \eta_s, f_t \eta_t], f_1 \eta_1, \ldots, \widehat{f_s \eta_s}, \ldots, \widehat{f_t \eta_t}, \ldots, f_{k+1} \eta_{k+1}) \\
&= \sum_{1 \leq t \leq k+1} (-1)^t f_t \, \eta_t \left( f_1 \ldots \hat{f}_t \ldots f_{k+1} \, \vphi(\eta_1, \ldots, \widehat{\eta_t}, \ldots, \eta_{k+1}) \right)\\
& \quad + \sum_{1 \leq s < t \leq k+1} (-1)^{s+t-1} \vphi( f_s f_t [\eta_s, \eta_t], f_1 \eta_1, \ldots, \widehat{f_s \eta_s}, \ldots, \widehat{f_t \eta_t}, \ldots, f_{k+1} \eta_{k+1}) \\
& \quad + \sum_{1 \leq s < t \leq k+1} (-1)^{s+t-1} \vphi( f_s \eta_s(f_t) \eta_t, f_1 \eta_1, \ldots, \widehat{f_s \eta_s}, \ldots, \widehat{f_t \eta_t}, \ldots, f_{k+1} \eta_{k+1}) \\
& \quad - \sum_{1 \leq s < t \leq k+1} (-1)^{s+t-1} \vphi( f_t \eta_t(f_s) \eta_s, f_1 \eta_1, \ldots, \widehat{f_s \eta_s}, \ldots, \widehat{f_t \eta_t}, \ldots, f_{k+1} \eta_{k+1}) \\
&= \sum_{1 \leq t \leq k+1} (-1)^t f_1 \ldots f_{k+1} \, \eta_t \, \vphi(\eta_1, \ldots, \widehat{\eta_t}, \ldots, \eta_{k+1})\\
& \quad + \sum_{1 \leq t \leq k+1} (-1)^t f_t \,\eta_t (f_1 \ldots \hat{f}_t \ldots f_{k+1}) \, \vphi(\eta_1, \ldots, \widehat{\eta_t}, \ldots, \eta_{k+1})\\
& \quad + \sum_{1 \leq s < t \leq k+1} (-1)^{s+t-1} f_1 \ldots f_{k+1} \, \vphi( [\eta_s, \eta_t], \eta_1, \ldots, \widehat{\eta_s}, \ldots, \widehat{ \eta_t}, \ldots, \eta_{k+1}) \\
&\quad + \sum_{1 \leq t \leq k+1} \sum_{s \neq t} (-1)^{t-1} \eta_t (f_s) f_1 \ldots \hat{f}_s \ldots f_{k+1} \, \vphi(\eta_1, \ldots, \widehat{\eta_t}, \ldots, \eta_{k+1})\\
&= f_1 \ldots f_{k+1} \,d \vphi(\eta_1, \ldots, \eta_{k+1}).
\end{align*}    
\end{proof}

Consider now $A$-linear cohomology of $A \sp V$.
By definition, the cochains $C_A^k(A \sp V, M)$ can be viewed as the cochains $\vphi \in C^k(A \sp V, M)$ such that 
\[
\vphi(f_1 \sp \eta_1, \dots, f_k \sp \eta_k) = f_1 \cdots f_k\,\vphi(1\sp\eta_1, \dots, 1\sp\eta_k).
\] 
Since $1 \sp V$ is a subalgebra of $A \sp V$, we can restrict any cochain $\psi \in C_A^k(A \sp V, M)$ to a cochain in $C^k(1 \sp V, M)$. 
This gives us the restriction map
\begin{align*}
\overline{\,\cdot\,}: C_A^k(A \sp V,M) \rightarrow  C^k(V,M).
\end{align*}
We can also lift cochains from $C^k(V,M)$ to $C_A^k(A \sp V,M)$ by $A$-linearity, yielding the map
\begin{align*}
\widetilde{\cdot}: C^k(V,M) \rightarrow  C_A^k(A \sp V,M),
\end{align*}
defined by $\hphi(f_1\sp \eta_1, \dots, f_k \sp \eta_k) = (f_1 \cdots f_k) \,\vphi(\eta_1, \dots, \eta_k)$. 

\begin{proposition}
\label{Alinearity}
The cochain complexes  $C^*(V,M)$ and $C_A^*(A \sp V,M)$ are isomorphic with 
the isomorphisms given by the lift and restriction maps.
\end{proposition}
\begin{proof}
It is clear that the lift and restriction maps are the inverses of each other.
We need to prove that these maps commute with the differentials. Since the lift and the restriction are the inverses of each other, it is sufficient to check this property for either one of them. Checking that 
$\overline{d \psi} = d \, \overline{\psi}$ is completely straightforward and we omit the details.
\end{proof}

\begin{corollary}\label{corollary:tildebijective}
For a Rinehart pair $(A, V)$ and an $AV$-module $M$, we have an isomorphism of cohomologies 
$H^*(V, M) \cong H_A^*(A \sp V, M)$.  
\end{corollary}

Thus by studying the $A$-linear cohomology of the Lie algebra $A \sp V$, we get the cohomology of $V$. 

Next, let us see what is the image of the Gelfand-Fuks subcomplex in $C_A^*(A \sp V, M)$. Let $\vphi$ be a Gelfand-Fuks cochain. Then its lift satisfies the identity
\[
\hphi \left( 
 g_1 \sum_{i=0}^p (-1)^i \binom{p}{i} f^{p-i} \# f^i \eta_1, \, g_2 \# \eta_2, \dots, g_k \# \eta_k
\right) = 0,
\]
so $\hphi$ vanishes when one of its arguments belongs to $\Delta^p \otimes_A V$.

\begin{definition}
A cochain $\psi \in C_A^k (A \sp V, M)$ is called \emph{finite} if there exists $p \in \N$ such that $\psi$ vanishes when one of its arguments belongs to the ideal $\Delta^p \otimes_A V$.
\end{definition}

We conclude that the lifts of Gelfand-Fuks cochains from $C^k(V,M)$ to $C_A^k(A \sp V,M)$ are precisely finite cochains. Thus finite cochains form a subcomplex $C_{A,\text{fin}}^*(A \sp V,M)$, and this complex is isomorphic 
to the Gelfand-Fuks complex $C_\text{GF}^*(V,M)$.

Finite cochains may be extended to the completed algebra $A \widehat{\sp} V$ 
(Definition \ref{definition:algebrajets}).
We call a cochain $\psi$ in $C_A^k (A \widehat{\sp} V, M)$ finite if there exists $p \in \N$ such that $\psi$ vanishes when
one of its arguments is of the form $\sum_{i=p}^\infty x_i$ 
with $x_i \in \Delta^i \otimes_A V$. We immediately see that complexes
$C_{A,\text{fin}}^*(A \sp V,M)$ and $C_{A,\text{fin}}^*(A \widehat{\sp} V,M)$
are isomorphic. As a consequence we get the following result:
\begin{proposition}
\label{isoGF}
Let $M$ be an $N$-differentiable $AV$-module. Then
\[ 
H_\text{GF}^*(V,M) \cong H_{A, \text{fin}}^* (A \sp V, M)
\cong H_{A, \text{fin}}^* (A \widehat{\sp} V, M)
\cong \varinjlim\limits_{k \geq N} H_{A}^* (A \sp V / (\Delta^k \otimes_A V), M).
\]
\end{proposition}


The Lie algebra of jets of vector fields $A \widehat{\sp} V$ has a particularly nice description when the variety $X$ has uniformizing parameters, that is, 
$X$ may be covered by a single \'etale chart. Let us recall this construction
(see e.g. \cite{BilligIngalls} for details).

For a variety $X$, denote by $\A$ the sheaf of regular function on $X$ and $\V$ the sheaf of vector fields. 

\begin{definition}
Let $X$ be an affine variety. An \emph{\'{e}tale chart} on $X$ is a Zariski open affine subset $U$ of $X$ such that there exist functions $x_1, \dots, x_n \in \A(U)$ satisfying
\begin{enumerate}[label=(\roman*)]
\item $\{x_1, \dots, x_n\}$ is algebraically independent,
\item Every $f \in \A(U)$ is algebraic over $\k[x_1, \dots, x_n]$,
\item The derivations $\dx{1}, \dots, \dx{n}$ of $\k[x_1, \dots, x_n]$ can be extended to $\A(U)$.
\end{enumerate}
We will call the set $\{x_1, \dots, x_n\}$ the \emph{uniformizing parameters of $U$}. 
\end{definition}
Note that such extensions of $\dx{i}$ are unique and commute with each other. 


For the rest of the section, let $U$ be a \'etale chart and set $A = \A(U)$ and $V = \V(U)$. Then $(A, V)$ is a Rinehart pair.

To see the connection of the completions $A \widehat{\otimes} A$ and
$A \widehat{\sp} V$ to the theory of jets, consider
a homomorphism of commutative algebras $j: A \otimes_\k A \rightarrow A \otimes\k\llbracket t_1, \dots, t_n \rrbracket$ with
\[
j(g \otimes f)  = \sum_{m \in \Z_{+}^n} \frac{1}{m!} g(x) \frac{\partial^m f}{\partial x^m} t^m
\]
(\cite[Lemma 10]{BilligIngalls}). 
The image $j(1 \otimes f)$ is the \emph{jet of the function $f$}. 
In fact, $j$ extends to an isomorphism of commutative algebras $A \widehat{\otimes} A$ and $A \otimes \k \llbracket t_1, \dots, t_n \rrbracket$ \cite[Corollary 18]{BilligIngalls}. 

To give a realization of $A \widehat{\sp} V$, we need to introduce another Lie algebra, $\L_+$. Denote $\L = \Der(\k[X_1, \dots, X_n])$ with the natural $\Z$-grading by total degree. Let $\L_k$ be the $k^\text{th}$ graded component, so that
\[
\L = \L_{-1} \oplus \L_0 \oplus \L_1 \oplus \L_2 \oplus \cdots .
\]
We set $\L_+ = \bigoplus_{i=0}^\infty \L_i$ be the positive part of $\L$ and note that $\L_0 \cong \gl_n$. 
Set $\L_{\geq k} = \bigoplus_{i=k}^\infty \L_i$. We will need a completion of the tensor product:
\[
A \wo \L_+ = \varprojlim\limits_k  A \otimes (\L_+/\L_{\geq k}).
\]


We now state the decomposition of $A \widehat{\sp} V$ proved in \cite{BilligIngalls}. 

\begin{theorem}[{\cite[Theorem 22, Remark 23]{BilligIngalls}}]\label{theorem:BItheorem22}
Let $U$ be an \'etale chart of $X$ with uniformizing parameters $x_1, \dots, x_n$. Let $A = \A(U)$ and $V = \V(U)$, then
\[
A \widehat{\sp} V \cong V \ltimes (A \wo \L_+),
\]
\[
A \sp V / (\Delta^{k+1} \otimes_A V) \cong V \ltimes (A \otimes \L_+/\L_{\geq k}).
\]
When $U \cong \affine^n$, this isomorphism holds without needing completions:
\[
A \sp V \cong V \ltimes (A \otimes \L_+).
\]
\end{theorem}

We will use this result to compute Gelfand-Fuks cohomology for varieties with a single \'etale chart. To do this, we will need to establish a K\"unneth formula for the semi-direct product  $V \ltimes (A \otimes \L_+/\L_{\geq k})$.  

In Theorem \ref{theorem:intro}, one of the terms in the decomposition is the de Rham cohomology. We revisit this cohomology and show that for the vector fields on an algebraic variety, de Rham cohomology is exactly $A$-linear cohomology of $V$ with coefficients in $A$. 

\section{de Rham Cohomology of an affine variety}\label{section:deRham}

In this section, we will show that the de Rham cohomology of an affine algebraic variety can be realized as the $A$-linear Lie algebra cohomology of vector fields with coefficients in the algebra of functions $A$. For this section, we assume that $X$ is a smooth irreducible affine algebraic variety over $\k$. Set $A$ be the coordinate ring of $X$ and set $V$ be the vector fields on $X$. 

First, we recall the de Rham cohomology of an algebraic variety. View $A \otimes_\k A$ as an $A$-module with the action on the left tensor factor and consider a submodule $I = \langle 1 \otimes fg  - f \otimes g - g \otimes f \,\big|\, f,g \in A\rangle$. Set $\Omega^1 = A \otimes_\k A /I$ and for $f, g \in A$ we use the notation $f d(g)$ for the image of $f \otimes g$ in $\Omega^1$. Let $\Omega^k (A) = \bigwedge^k_A \Omega^1 (A)$. As we take this alternating product over $A$, elements in $\Omega^k$ are written as $f dg_1\wedge \dots \wedge dg_k$ for $f, g_1, \dots, g_k \in A$. We form a complex by defining the de Rham differential $d_{\DR}: \Omega^k(A) \rightarrow \Omega^{k+1}(A)$ by
\[
d_{\DR}(f dg_1 \wedge \cdots \wedge d g_k) = df \wedge dg_1 \wedge \cdots \wedge dg_k. 
\]
for $f \in A, dg_1, \dots, dg_k \in \Omega^1$. Set $\Omega^{k}=0$ for $k < 0$ and $\Omega^0 =A$, then the de Rham cohomology $H^*_{\DR}(X)$ is the cohomology of the de Rham complex $\{ \Omega^i, d_{\DR}\}$:
\[
0 \rightarrow A \rightarrow \Omega^1 \rightarrow \Omega^2 \rightarrow \cdots.
\] 


\begin{theorem}\label{theorem:derhamcohomology}
Let $X$ be a smooth irreducible affine algebraic variety, let $A$ be the coordinate ring of $X$ and let $V$ be the Lie algebra of vector fields on $X$. Then the cochain complex $C_A^* (V, A)$ is isomorphic to the de Rham complex $\Omega^*_X$, and
\[
H^*_A(V,A) \cong H^*_{\DR}(X).
\]
\end{theorem}

\begin{proof}
We need to show that the two complexes are isomorphic. Consider the map $\Phi: \Omega^k(X) \rightarrow \Hom_A \left( \bigwedge^k_A V, A \right)$ by
\[
\Phi(f dg_1 \wedge \cdots \wedge dg_k) (\eta_1 \wedge \cdots \wedge \eta_k) = f \det (\eta_j(g_i))
\]
for $f, g_1, \dots, g_k \in A, \, \eta_1, \dots, \eta_k \in V$. 
We point out that for $\mu \in V$,
\[ 
\mu \left( \Phi(dg_1 \wedge \cdots \wedge dg_k) (\eta_1 \wedge \cdots \wedge \eta_k) \right) =
\sum_{t=1}^k \det \left[ \eta_1(g_j) \cdots \mu(\eta_t (g_j)) \cdots \eta_k(g_j) \right]_{j=1, \ldots, k}.
\]
This equality follows from the fact that $\mu$ is a derivation of $A$. Using this equality, we can show that 
$\Phi$ is a homomorphism of $AV$-modules. Since $\Phi$ is manifestly $A$-linear, we need to check that
\[
\mu \left( \Phi(f dg_1 \wedge \cdots \wedge dg_k) \right) = \Phi \left( \mu(f dg_1 \wedge \cdots \wedge dg_k) \right). 
\]
Indeed,
\begin{align*}
&\mu \left( \Phi(f dg_1 \wedge \cdots \wedge dg_k) \right) (\eta_1 \wedge \cdots \wedge \eta_k) \\
&= \mu \left( f \det (\eta_j(g_i)) \right) - \sum_{t=1}^k  \Phi(f dg_1 \wedge \cdots \wedge dg_k) 
(\eta_1 \wedge \cdots \wedge [\mu, \eta_t] \wedge \cdots \wedge \eta_k) \\
&= \mu(f) \det (\eta_j(g_i)) 
+ f \sum_{t=1}^k \det \left[ \eta_1(g_j) \cdots \mu(\eta_t (g_j)) \cdots \eta_k(g_j) \right]_{j=1, \ldots, k} \\
&-  f \sum_{t=1}^k \det \left[ \eta_1(g_j) \cdots \mu(\eta_t (g_j)) \cdots \eta_k(g_j) \right]_{j=1, \ldots, k}
+  f \sum_{t=1}^k \det \left[ \eta_1(g_j) \cdots \eta_t(\mu (g_j)) \cdots \eta_k(g_j) \right]_{j=1, \ldots, k} \\
&=  \Phi \left( \mu(f dg_1 \wedge \cdots \wedge dg_k) \right) (\eta_1 \wedge \cdots \wedge \eta_k) .
\end{align*}
Next, we are going to show that $d_A \circ \Phi = \Phi \circ (-d_{dR})$. We have
\begin{align*}
\Phi(d_{\DR}(g_1 dg_2 \wedge \cdots \wedge dg_{k+1}))(\eta_1\wedge \cdots \wedge \eta_{k+1}) & = \det (\eta_j(g_{i})).
\end{align*}
 Let $M_{i,j}$ be the determinant of the minor obtained by deleting the $i^\text{th}$ row and the $j^\text{th}$ column in the $(k+1) \times (k+1)$ matrix $(\eta_j (g_i))$. 
  By the cofactor expansion of the determinant along the first row,
\[
\det (\eta_j(g_{i})) = \sum_{j=1}^{k+1} \eta_j(g_1) (-1)^{j+1} M_{1,j}.
\]
If we apply the differential of the Lie algebra cohomology, 
\begin{align*}
& \quad d(\Phi(g_1 dg_2 \wedge \cdots \wedge dg_{k+1}) )(\eta_1 \wedge \cdots \wedge \eta_{k+1}) \\
& = \sum_{1 \leq t \leq k+1} (-1)^t \eta_t \Phi(g_1 dg_2 \wedge \cdots \wedge dg_{k+1})(\eta_1 \wedge \cdots \wedge \hat{\eta}_t \wedge \cdots \wedge \eta_{k+1})\\
& \quad + \sum_{1 \leq s < t \leq k+1} (-1)^{s+t-1} \Phi(g_1 dg_2 \wedge \cdots \wedge dg_{k+1}) ([\eta_s, \eta_t] \wedge \eta_1 \wedge \cdots \wedge \hat{\eta}_s \wedge \cdots \wedge \hat{\eta}_t \cdots \wedge \eta_{k+1})\\
& = \sum_{1 \leq t \leq k+1} (-1)^t \eta_t \left(g_1 M_{1,t} \right) \\ 
& + \sum_{1 \leq s < t \leq k+1} (-1)^{s+t-1} g_1 \det\bigg[ [\eta_s, \eta_t](g_j) \ \eta_1(g_j) \ \ldots \ \hat\eta_{s}(g_j) \ \ldots \ \hat{\eta}_{t}(g_j) \ \ldots \ \eta_{k+1}(g_j)\bigg]_{j = 2, \ldots, k+1}  \\
& = \sum_{1 \leq t \leq k+1} (-1)^t \eta_t \left(g_1 M_{1,t} \right) \\ 
& + \sum_{1 \leq s < t \leq k+1} (-1)^{s-1} g_1 \det\bigg[ \eta_1(g_j) \ \ldots \  \hat\eta_{s}(g_j) \   \ldots \ \eta_{t-1}(g_j)  \ \eta_s (\eta_t(g_j)) \ \eta_{t+1}(g_j) \ \ldots \ \eta_{k+1}(g_j)\bigg]  \\
&+ \sum_{1 \leq s < t \leq k+1} (-1)^{t-1} g_1 \det\bigg[ 
\eta_1(g_j) \ \ldots \ \eta_{s-1}(g_j) \ \eta_t (\eta_s (g_j)) \ \eta_{s+1}(g_j) \ \ldots \ \hat{\eta}_{t}(g_j)  \ \ldots \ \eta_{k+1}(g_j) \bigg]\\
& = \sum_{1 \leq t \leq k+1} (-1)^t \eta_t \left(g_1 M_{1,t} \right) \\ 
&+ \sum_{1 \leq t \leq k+1} \sum_{s \neq t} (-1)^{t-1} g_1 \det\bigg[ 
\eta_1(g_j) \ \ldots \ \eta_{s-1}(g_j) \ \eta_t (\eta_s (g_j)) \ \eta_{s+1}(g_j) \ \ldots \ \hat{\eta}_{t}(g_j)  \ \ldots \ \eta_{k+1}(g_j) \bigg]\\
& = \sum_{1 \leq t \leq k+1} (-1)^t \eta_t \left(g_1 M_{1,t} \right)
+ \sum_{1 \leq t \leq k+1} (-1)^{t-1} g_1 \eta_t \left( M_{1,t} \right) \\
& = \sum_{1 \leq t \leq k+1} (-1)^t \eta_t(g_1) M_{1,t} =  - \det(\eta_j(g_i))\\
& = \Phi(-d_{\DR}(g_1 dg_2 \wedge \cdots \wedge dg_k))(\eta_1\wedge \cdots \wedge \eta_{k+1}).
\end{align*}

We now show that $\Phi$ is an isomorphism. First, consider $\vphi \in \Hom_A(\Omega^1(A), A)$. Since $\Omega^1(A) = A \otimes A / \langle 1 \otimes fg - f \otimes g - g \otimes f \rangle$ as a left-module, then $\vphi$ is determined by its action on $1 \otimes A$. Set $\psi(f) = \vphi(1 \otimes f)$ so that $\psi \in \Hom(A, A)$. Since $\vphi $ is zero on the ideal $\langle 1 \otimes fg - f \otimes g - g \otimes f \rangle$, then $\psi(fg) = f \psi(g) + g \psi(f)$ so $\psi \in \Der(A) = V$. Thus $\Hom_A(\Omega^1, A) \cong V$.

For an $AV$-module $M$, denote by $M^*$ its dual, $M^* = \Hom_A(M, A)$.
 As $V$ and $\Omega^1$ are $AV$-modules which are finitely generated over $A$, then by \cite[Proposition 1.2]{BouazizRocha} they are also projective $A$-modules. Thus by \cite{BourbakiAlgebra}, these modules are both isomorphic to their double dual:
\[
V^{**} \cong V, \quad\quad (\Omega^1)^{**} \cong \Omega^1.
\]
In particular, $\Omega^1 \cong V^*$.

Now by definition, $\Omega^k(A) = \bigwedge^k_A \Omega^1(A)$,
and the alternating product commutes with the dual:
\[
\left( \bigwedge\nolimits^k_A V \right)^* \cong \bigwedge\nolimits^k_A V^*.
\]
Thus $\bigwedge^k_A \Omega^1(A) \cong \bigwedge^k_A V^* \cong \left( \bigwedge^k_A V\right)^*$ and hence $\Omega^k(A) \cong \Hom_A(\bigwedge^k_A V, A)$ so that the two complexes are isomorphic and hence the cohomologies are equal. 
\end{proof}

\section{A K\"unneth formula for a semi-direct product}\label{section:Kunneth}

A K\"unneth formula is the following result about the cohomology of a direct sum:

\begin{theorem}
\label{direct_sum_Kunneth}
Let $\g_1, \g_2$ be Lie algebras, and let $W_1, W_2$ be modules for $\g_1$ and $\g_2$ respectively. Then 
\[
 H^n(\g_1 \oplus \g_2, W_1 \otimes W_2) \cong \sum_{i+j=n} H^i(\g_1, W_1) \otimes H^{j}(\g_2, W_2).
\]
\end{theorem}

We shall also write the above isomorphism as 
$H^*(\g_1 \oplus \g_2, W_1 \otimes W_2) \cong  
H^*(\g_1, W_1) \otimes H^*(\g_2, W_2)$.

For our purposes, we will need a generalization of the K\"unneth formula for semi-direct products of a particular kind, which can be applied to the Lie algebra
$V \ltimes (A \otimes \widehat{\L}_+)$. 

Let $\g$ be a Lie algebra, and let $W$ be a $\g$-module. If $A$ is a commutative algebra then $A \otimes W$ is naturally a module for the Lie algebra 
$A \otimes \g$. If $(A,V)$ is a Rinehart pair, then we can form a semi-direct product Lie algebra  $V \ltimes (A \otimes \g)$, where $V$ acts on $A$. 
Let $S$ be a strong $AV$-module.
Then $S \otimes W$ is a module for this semi-direct product, with the action
\begin{align*}
v \cdot (m \otimes w) = (v \cdot m) \otimes w, \\
(f \otimes x) \cdot (m \otimes w) = fm \otimes xw,
\end{align*}
where $v \in V, f \in A, x \in \g, m \in S, w \in W$.
Note that $(A, V \ltimes (A \otimes \g))$ is a Rinehart pair, with a trivial action of $A \otimes \g$ on $A$, and $S \otimes W$ is a strong module for this pair. Thus we can define $A$-linear cohomology of $V \ltimes (A \otimes \g)$ with values in $S \otimes W$.

\begin{theorem}
\label{semidirect_Kunneth}
Let $(A, V)$ be a Rinehart pair, let $S$ be a strong $AV$-module, and let $W$ be a module for the Lie algebra $\g$. Then
\[
H^*_A (V \ltimes (A \otimes \g),\, S \otimes W) \cong H^*_A (V, S) \otimes H^*(\g, W).
\]
\end{theorem}

The proof is based on the Hochschild-Serre spectral sequence \cite{HochschildSerre1953}. We consider a filtration on the space of cochains $C^k = C^k_A (V \ltimes (A \otimes \g),\, S \otimes W)$:
\[
C^k = F^0 C^k \supset F^1 C^k \supset \ldots \supset F^k C^k \supset F^{k+1} C^k = (0),  
\]
where 
\[
F^pC^{p+q} = \left\{ \vphi \in C^{p+q} \, \big| \, \vphi(x_1, \ldots, x_{p+q}) = 0 \ \text{for } x_1, \ldots, x_{q+1} \in A \otimes \g \right\}.
\]
It is easy to check that $d F^pC^{p+q} \subset F^pC^{p+q+1}$.
Let $\{ E_r^{p,q}, d_r^{p,q} \}$ be the spectral sequence associated with this
filtration (see \cite{McCleary2001} for details on spectral sequences).

The $E_0$ term of this spectral sequence is defined as
\[
E_0^{p,q} = F^pC^{p+q} / F^{p+1}C^{p+q},
\]
and the differential $d_0^{p,q} : \, E_0^{p,q} \rightarrow E_0^{p,q+1}$ is induced by the differential of the cochain complex.

The $E_2$ term of this spectral sequence was computed in \cite{HochschildSerre1953} (see also Theorem 1.5.1 in \cite{Fuks1986}):
\[
E_2^{p.q} = H_A^p (V, H_A^q (A \otimes \g,\, S \otimes W)).
\]
The argument used in the proof of Proposition \ref{Alinearity} shows that the complexes $C^*_A(A \otimes \g,\, S \otimes W)$ and $S \otimes C^*(\g, W)$
are isomorphic. Thus,
\[
 H_A^q (A \otimes \g,\, S \otimes W) \cong S \otimes H^q (\g, W).
\]
Since the action of $V$ on $H^q (\g, W)$ is trivial, we get that
\[
E_2^{p.q} =  H_A^p (V, S \otimes H^q (\g, W)) \cong H_A^p (V, S) \otimes H^q (\g, W).
\]
To complete the proof of the theorem, we need to show that the differentials $d_r^{p,q}$ vanish for $r\geq 2$, and the spectral sequence stabilizes at $E_2$ term.

 We construct a map between complexes
 \[
 \ast: \, C^*_A (V, S) \otimes C^* (\g, W) \rightarrow C^*_A (V \ltimes (A \otimes \g),\, S \otimes W).
 \]
 Define $\ast:\, C_A^k(V, S) \otimes C^m(\g, W) \rightarrow C_A^{k+m}(V \ltimes (A \otimes \g), \, S \otimes W)$ by $(\alpha, \beta) \mapsto \alpha \ast \beta$, where 
\begin{equation}\label{equation:starmap}
\alpha \ast \beta (v_1, \dots, v_k, f_1 \otimes x_1, \dots, f_m \otimes x_m) = f_1 \cdots f_m \alpha(v_1, \dots, v_k) \otimes \beta(x_1, \dots, x_m)
\end{equation}
for $\alpha \in C_A^k(V, S), \beta \in C^m(\g, W)$. 
Here we assume that every argument of $\alpha \ast \beta$ is either in $V$ or in $A \otimes \g$. If the number of arguments of $\alpha \ast \beta$ from $V$ is not equal to $k$, we set the value of $\alpha \ast \beta$ to be zero.

\begin{lemma}
For any $\alpha \in C_A^k(V,S)$ and $\beta \in C^m(\g, W)$ 
\[
d(\alpha \ast \beta) = d \alpha \ast \beta + (-1)^k\alpha \ast d \beta.
\]
\end{lemma}

\begin{proof}
By definition, $\alpha \ast \beta \in C^{k+m}(V \ltimes (A \otimes \g),\, S \otimes W)$. By the choice of $\alpha$ and $\beta$, $\alpha \ast \beta$ will only be nonzero on $\bigwedge^k V \wedge \bigwedge^m (A \otimes \g)$ and so $d(\alpha \ast \beta)$ will only be nonzero on $\bigwedge^{k+1} V \wedge \bigwedge^m (A \otimes \g) + \bigwedge^k V \wedge \bigwedge^{m+1} (A \otimes \g)$. We will calculate $d(\alpha \ast \beta)$ on each summand. 

For $v_0, v_1, \dots, v_k \in V$ and $f_1 x_1, \dots, f_m  x_m \in A\otimes \g$, we will use the short hand of $\beta(x)$ for $\beta(x_1, \dots, x_m)$.
\begin{align*}
& \quad d (\alpha \ast \beta) (v_0, v_1, \dots, v_k, f_1  x_1, \dots, f_m  x_m) \\
&= \sum_{0 \leq s< t\leq k} (-1)^{s+t-1} f_1 \cdots f_m \alpha ([v_s, v_t] , v_0, \dots, \hat{v}_s, \dots, \hat{v}_t, \dots, v_k) \otimes \beta(x) \\
&\quad+ \sum_{0 \leq s \leq k} (-1)^{s+1} v_s \cdot f_1 \cdots f_m \alpha(v_0, \dots, \hat{v}_s, \dots, v_k)  \otimes \beta(x)\\
&\quad+\sum_{\substack{0 \leq s \leq k, \\ 1 \leq t\leq m}} (-1)^{s+t-1} (-1)^{k+t-1} v_s(f_t) f_1 \cdots \hat{f_t} \cdots f_m \alpha(v_0, \dots, \hat{v}_s, \dots, v_k) \otimes \beta (x) \\
&= \sum_{0 \leq s< t\leq k} (-1)^{s+t-1} f_1 \cdots f_m \alpha ([v_s, v_t] , v_0, \dots, \hat{v}_s, \dots, \hat{v}_t, \dots, v_k) \otimes \beta(x) \\
&\quad+ \sum_{0 \leq s \leq k} (-1)^{s+1}  f_1 \cdots f_m \cdot v_s \cdot \alpha(v_1, \dots, \hat{v}_s, \dots, v_k)  \otimes \beta(x) \\
& \quad+ \sum_{\substack{0 \leq s \leq k, \\ 1 \leq t\leq m}} (-1)^{s+1} v_s(f_t) \cdot f_1 \cdots \hat{f_t} \cdots f_m \alpha(v_1, \dots, \hat{v}_s, \dots, v_k)  \beta(x) \\
& \quad+\sum_{\substack{0 \leq s \leq k, \\ 1 \leq t\leq m}}(-1)^{s} v_s(f_t) \cdot f_1 \cdots \hat{f_t} \cdots f_m \alpha(v_0, \dots, \hat{v}_s, \dots, v_k) \otimes \beta (x) \\
& =   ((d\alpha)\ast \beta) (v_0, v_1, \dots, v_k, f_1  x_1, \dots, f_m  x_m) 
\end{align*}

Now for $v_1, \dots, v_k \in V$ and $f_0  x_0, \dots f_m  x_m \in A \otimes \g$, let $\alpha(v)$ be shorthand for $\alpha(v_1, \dots, v_k)$. 
\begin{align*}
& \quad (-1)^k d(\alpha \ast \beta) (f_0  x_0, v_1, \dots, v_k, f_1  x_1, \dots, f_m  x_m) \\
&= d(\alpha \ast \beta) (v_1, \dots, v_k, f_0  x_0, f_1  x_1, \dots, f_m  x_m) \\
&= \sum_{0 \leq s< t\leq m} (-1)^{s+t-1}(-1)^{k} \alpha \ast \beta (v, [f_s  x_s, f_t  x_t],f_0  x_0, \dots, \widehat{f_s  x_s}, \dots, \widehat{f_t  x_t}, \dots, f_m  x_m ) \\
& \quad + \sum_{0 \leq s \leq m} (-1)^{s+k+1} (f_s  x_s) f_0 \cdots \hat{f}_s \cdots f_m \alpha(v) \otimes \beta( x_0,  \dots, \widehat{ x_s}, \dots,  x_m)\\ 
&= \sum_{0 \leq s< t\leq m} (-1)^{s+t+k-1} f_0 \cdots \hat{f_s} \cdots \hat{f_t} \cdots f_m \alpha(v) \otimes \beta ([ x_s, x_t],  x_0, \dots, \widehat{x_s}, \dots, \widehat{x_t}, \dots, x_m ) \\
&\quad + \sum_{0 \leq s \leq m} (-1)^{s+k+1} f_0 \cdots f_m \alpha(v) \otimes x_s\beta( x_0,  \dots, \widehat{ x_s}, \dots,  x_m)\\ 
& = (-1)^k (\alpha \ast d \beta) (v_1, \dots, v_k, f_0  x_0, f_1 x_1, \dots, f_m  x_m)
\end{align*}
Thus, $d(\alpha \ast \beta) = d\alpha \ast \beta + (-1)^k \alpha \ast d \beta$. 
\end{proof}

\begin{proposition}
\label{cohomology_map}
The map $\ast$ induces a map of cohomology spaces:
\[
\mathop\oplus\limits_{i+j=n} H_A^i(V, S) \otimes H^{j}(\g, W) \rightarrow H_A^{n}(V \ltimes (A \otimes \g),\, S \otimes W).
\]
\end{proposition}
\begin{proof}
We use the standard notation $Z^i$ and $B^i$ to denote the kernels and images, respectively, of the differential maps $d^i$. 

The restriction of the map $\ast$ to cocycles will have codomain $Z_A(V\ltimes (A \otimes \g),\, S \otimes W)$ as $d(\alpha \ast \beta) =0$ if both $d\alpha =0$ and $d\beta =0$. To induce a map on cohomologies, we need the images of $B_A^i(V, S) \otimes Z^{j}(\g,W)$ and $Z_A^i(V,S) \otimes B^{j}(\g,W)$ to be contained in $B_A^n(V \ltimes (A \otimes \g),\, S \otimes W)$. Suppose $\alpha \in B_A^i(V,S)$ then $\alpha = d\delta$ for some $\delta \in C_A^{i-1}(V, S)$. Then $d(\delta \ast \beta) = d\delta \ast \beta= \alpha \ast \beta$ for every $\beta \in Z^{j}(\g, W)$. Similarly, $Z_A^i(V, S) \ast B^{j}(\g,W) \in B^n_A(V \ltimes(A \otimes \g),\, S \otimes W)$.  This completes the proof of the proposition.
\end{proof}

Returning to the proof of Theorem \ref{semidirect_Kunneth}, it follows from 
Proposition \ref{cohomology_map} that the elements of $E_2^{p,q}$ may be represented by cocycles. Since the differential $d_2^{p,q}$ is induced by the differential of the cochain complex $C^*_A (V \ltimes (A \otimes \g),\, S \otimes W)$, we get that
$d_2^{p,q} = 0$. By induction, we get that $E_r^{p,q} = E_2^{p,q}$ and 
$d_r^{p,q} = 0$ for all $r \geq 2$. Thus the spectral sequence stabilizes at the $E_2$ term and we get
\[
H^n (V \ltimes (A \otimes \g),\, S \otimes W) \cong \mathop\oplus\limits_{p+q = n} E^{p,q}_\infty \cong \mathop\oplus\limits_{p+q = n} H_A^p(V, S) \otimes H^q(\g, W).
\]
This completes the proof of Theorem \ref{semidirect_Kunneth}.

\begin{remark}
We were not able to find Theorem \ref{direct_sum_Kunneth} in the literature.
Instead, we get it as a corollary to Theorem \ref{semidirect_Kunneth} by setting $A = \k$.
\end{remark}

When $V$ is the Lie algebra of vector fields on an \'etale chart and $A$ is the coordinate ring, we can apply this K\"unneth formula to compute Gelfand-Fuks cohomology of vector fields with values in tensor modules. 

\section{Gelfand-Fuks cohomology for varieties with uniformizing parameters}\label{section:etalecharts}

In this section we will summarize our results in the case when variety $X$ has uniformizing parameters, i.e., it is covered by a single \'etale chart.
Such varieties are smooth.
Following the notations from Section \ref{section:background}, let $\{x_1, \dots, x_n\}$ be the uniformizing parameters on $X$, 
$A$ be the algebra of functions on $X$ and $V$ be the Lie algebra of vector fields. 

In this case $V$ is a free $A$-module \cite{BilligFutornyNilsson}:
\[
V = \mathop\oplus\limits_{i=1}^n A \frac{\partial}{\partial x_i}.
\]
The associative algebra $D$ of differential operators on $X$ is also free over $A$:
\[
D = \mathop\oplus\limits_{k \in \Z_{+}^n} A \,\partial^k,
\]
where $\partial^k = \left(\frac{\partial}{\partial x_1}\right)^{k_1} \cdot \ldots \cdot \left(\frac{\partial}{\partial x_n}\right)^{k_n}$.

Let $\L_+$ be the non-negatively graded part of the Lie algebra $\Der (\k[X_1, \dots, X_n])$, and let $\widehat{\L}_+$ be its completion. From Theorem \ref{theorem:BItheorem22}, the Lie algebra of jets of vector fields decomposes as follows:
\begin{equation}
\label{iso}
A \widehat{\sp} V \cong V \ltimes (A \wo \L_+).
\end{equation}

Let $W$ be a finite-dimensional $\L_+$-module. By Lemma 2 in \cite{Billig2007}, graded components $\L_d$ act as zero on $W$ for $d$ large enough. 
Simple finite-dimensional $\L_+$-modules are precisely simple finite-dimensional modules for $\gl_n \cong \L_0$, with $\L_d$ acting trivially for $d \geq 1$.

Given a $D$-module $S$ and a finite-dimensional $\L_+$-module $W$,
the tensor product $M = S \otimes W$ is a module for the semi-direct product
$V \ltimes (A \wo \L_+$). Taking into account isomorphism (\ref{iso}), we conclude that $M$ is a module for $A \widehat{\sp} V$.
Note that in the isomorphism (\ref{iso}), elements of the form 
$\sum_{j=d}^\infty x_j$ with $x_j \in \Delta^{j+1} \otimes_A V$ are precisely those that are mapped into the space $A \wo  \L_{\geq d}$.
Thus for $d$ large enough, such elements annihilate $M$ and we conclude that $M$ is a differentiable $AV$-module.

Let us describe the structure of an $AV$-module on $M = S \otimes W$. Here $A$ acts on $S$ and the action of the vector fields is given as follows (cf. \cite{BilligFutornyNilsson}):
\[
f \frac{\partial}{\partial x_i} (s \otimes w) = \left( f \frac{\partial}{\partial x_i} s \right) \otimes w + \sum_{k \in \Z_{+}^n \backslash \{ 0 \}} \frac{1}{k!} \,\frac{\partial^k f}{\partial x^k}\, s \otimes \left( X^k \frac{\partial}{\partial X_i}\right) w,
\]
for $f \in A, s\in S, w \in W$. Note that the sum in the right-hand side is finite since elements of $\L_+$ of high enough degree annihilate $W$.

Let us mention three important classes of such modules $S \otimes W$.

1. Let $S = A$. In this case we obtain {\it tensor} modules $A \otimes W$. These generalize bundles of tensor fields on a variety and their bundles of finite jets.

2. Let $S$ be a free $A$-module of a finite-rank, $S = A \otimes U$, $\dim U < \infty$. Define gauge fields $B_i \in \End_A (A \otimes U)$, $i = 1, \dots, n$, satisfying 
\[
\left[ \frac{\partial}{\partial x_i} \otimes 1 + B_i, \, \frac{\partial}{\partial x_j} \otimes 1 + B_j \right] = 0.
\]
Then $S$ is a $D$-module with the action of $\frac{\partial}{\partial x_i}$
given by $\frac{\partial}{\partial x_i} \otimes 1 + B_i$.
The corresponding $AV$-module $S \otimes W$ is called a {\it gauge} module.

3. Let $S$ be a $D$-module of delta-functions, supported at a fixed point $P \in X$. Module $S$ is spanned by all partial derivatives of the delta-function $\delta_P$, so 
\[
S \cong \k \left[ \frac{\partial}{\partial x_1}, \ldots, \frac{\partial}{\partial x_n} \right] \delta_P .
\]
The corresponding $AV$-module $M = S \otimes W$ is called a {\it Rudakov} module (see \cite{BilligFutornyNilsson, BilligBouaziz} for details).

We are interested in the Gelfand-Fuks cohomology of $V$ with coefficients in an $AV$-module $M = S \otimes W$. By Proposition \ref{isoGF},
\[
H^*_\text{GF} (V, M) \cong H^*_{A, \text{fin}} (A \widehat{\sp} V, M)
\cong \varinjlim\limits_{k} H_{A}^* (A \sp V / (\Delta^k \otimes_A V), M).
\]
By Theorem \ref{theorem:BItheorem22}, Lie algebra $A \sp V / (\Delta^{k+1} \otimes_A V)$ decomposes into the semi-direct product $V \ltimes (A \otimes \L_+/\L_{\geq k})$, and $M = S \otimes W$. 
Our K\"unneth formula from Section \ref{section:Kunneth} states that 
\[
H^*_A (V \ltimes (A \otimes \L_+/\L_{\geq k}),\, S \otimes W) \cong H^*_A (V, S) \otimes H^*(\L_+/\L_{\geq k}, W).
\]
The $\Z$-grading on the Lie algebra $\L_+$ is given by the adjoint action of the element $X_1 \frac{\partial}{\partial X_1} + \ldots + X_n \frac{\partial}{\partial X_n}$. Module $W$ has a compatible grading by the generalized eigenvalues of this operator. By Theorem 1.5.2 in \cite{Fuks1986}, all cocycles in $H^*(\L_+, W)$ have degree zero. Since the grading on $W$ is finite, this implies that all cocycles in $H^*(\L_+, W)$ are finite and 
\[
H^*(\L_+, W) \cong H^*_\text{fin}(\L_+, W) \cong H^*(\L_+/\L_{\geq k}, W)
\]
for large enough $k$.
Our main result is the following theorem.
\begin{theorem}
\label{theorem:main}
(a) Let $X$ be an algebraic variety with uniformizing parameters. Let $S$ be a $D$-module, and let $W$ be a finite-dimensional module for the Lie algebra $\L_+$. Then the algebraic Gelfand-Fuks cohomology of the Lie algebra $V$ of vector fields on $X$ with values in $AV$-module $M = S \otimes W$ is given by
\[
H^*_\text{GF} (V, M) \cong H^*_A (V, S) \otimes H^*(\L_+, W).
\]
(b) In the case when $M$ is a tensor module, $S = A$, we have
\[
H^*_\text{GF} (V, A \otimes W) \cong H^*_\text{dR}(X) \otimes H^*(\L_+, W).
\]
\end{theorem}

\begin{remark}
The cohomology $H^*(\L_+, W)$ is known explicitly only in some special cases, see \cite{Fuks1986} for details.    
\end{remark}

\section{Applications}

We now consider a few special cases and explicitly compute these cohomologies using the de Rham cohomology and cohomology of $\L_+$. 

\subsection{Polynomial vector fields on affine space}\label{section:polynomial}

In the case of an affine space $X \cong \affine^n$, $\V(\affine^n)$ are polynomial vector fields, usually denoted by $W_n$. Gelfand's school studied the continuous cohomology of formal vector fields, $\widehat{W}_n$, under the projective limit topology, i.e. they consider continuous functionals from $\bigwedge^k \widehat{W}_n \rightarrow A$. It is naturally graded by $\Z$ so that $\widehat{\L}_+$, the Lie algebra of all the elements with strictly positive grading, is defined. 

\begin{theorem}[{\cite[Theorem 1.5.4]{Fuks1986}}]
\label{theorem:theorem1.5.4}
For $W$ a tensor $\gl_n$-module and $\Coind^{\widehat{W}_n}_{\widehat{\L}_+} W$ the corresponding space of formal tensor fields, there are natural isomorphisms
\[
H^*(\widehat{W}_n, \Coind_{\widehat{\L}_+} W) \cong H^* (\widehat{\L}_+, W).
\]
\end{theorem}
In our setting, we deal with polynomial vector fields $W_n$ rather than the completion $\widehat{W}_n$. In fact, in this case of affine space, we are able to work with the general cohomology theory, which turns out to coincide with the algebraic Gelfand-Fuks cohomology.

By Corollary \ref{corollary:tildebijective}, we have an isomorphism
\[
H^*(V, M) \cong H_A^*(A \sp V, M).
\]
The Lie algebra $A \sp V$ admits a decomposition into a semi-direct product without needing to pass to a completion:
\[
A \sp V \cong V \ltimes (A \otimes \L_+).
\]
Taking $M = A \otimes W$, we apply the K\"unneth formula (Theorem \ref{semidirect_Kunneth}) and we get
\[
H^*(V, M) \cong H^*_A (V, A) \otimes H^*(\L_+, W).
\]
Since $H^*_A (V, A) \cong H^*_\text{dR}(X)$ and the only non-zero de Rham cohomology of  affine space $\affine^n$ is $H^0_\text{dR} = \k$, we get the following corollary. 
\begin{corollary}
Let $X = \affine^n$. Let $W$ be a finite-dimensional $\L_+$-module, and let $M = A \otimes W$ be the corresponding tensor $AV$-module. Then 
\[
H^*(W_n, A \otimes W) \cong H^*(\L_+, W).
\]
\end{corollary}
This agrees with the topological result, Theorem \ref{theorem:theorem1.5.4}. 

\subsection{Polynomial vector fields on an $n$-dimensional torus}

In the case when $X \cong \T^n$, the Lie algebra of polynomial vector fields on a torus is closely related to structures of interest in physics. For example, in the case of $n=1$, a central extension of this algebra is the Virasoro algebra. The Virasoro and Witt algebras are the prototypical infinite-dimensional Lie algebras, and play an important role in conformal field theory and other areas. 

Set $V$ to be the polynomial vector fields on an $n$-dimensional torus $\T^n$. The algebra of functions is the algebra of Laurent polynomials $A =  \k [ x_1^{\pm 1}, \dots, x_n^{\pm 1} ]$. The Lie algebra $V$ of vector fields is a free $A$-module 
\[
V = \bigoplus_{i=1}^n A \dx{i},
\]
and $\{ x_1, \ldots, x_n\}$ are uniformizing parameters on the torus. 
This Lie algebra $V$ is graded by $\Z^n$, but unlike in the case of the affine space, this grading has a full support.

Let $W$ be a finite-dimensional $\L_+$-module, and consider a tensor module $M = A \otimes W$. By Theorem \ref{theorem:main},
\[
H_\GF^*(V, A \otimes W) \cong H_\DR^*(\T^n) \otimes H^*(\L_+, W).
\]
The de Rham cohomology in this case is generated by the cocycles $x_j^{-1}dx_j$ so that $H_\DR^i(\T^n) \cong \k^{\binom{n}{i}}$, which gives the following result:
\begin{corollary}
Let $X = \T^n$. Let $W$ be a finite-dimensional $\L_+$-module, and let $M = A \otimes W$ be the corresponding tensor $AV$-module. Then 
\[
H_\GF^k(V, A \otimes W) \cong \sum_{i+j =k} \k^{{n \choose i}} \otimes H^{j}(\L_+, W).
\]
\end{corollary}
\subsection{Krichever-Novikov algebras}

In general, Krichever-Novikov algebras are the algebras of vector fields on a Reimann surface with several punctured points. While the Virasoro algebra relates to conformal field theory of genus zero surfaces, these algebras relate to conformal field theory on higher genus surfaces. For an overview of Krichever-Novikov algebras, see \cite{Schlichenmaier2014}. In this section, we will consider the special case of meromorphic vector fields on a Riemann sphere with poles at $m+1$ prescribed points.

Let $X$ be the Riemann sphere punctured at points $\infty, a_1, \dots, a_m$, which we denote by 
$X = \overline{\C} \backslash\{\infty, a_1, \dots, a_m\}$. Then $X$ is a 1-dimensional affine variety over $\C$ with uniformizing parameter $z$, so that $A = \C[z, (z- a_1)^{-1}, \dots, (z-a_m)^{-1}]$ and $V = A \frac{\partial}{\partial z}$. 

For the de Rham cohomology of $X$ we have $H^0_\DR(X) = \C$, while $H^1_\DR(X) = \C^m$, and is spanned by $(z-a_i)^{-1} dz$. We obtain the following result:
\begin{corollary}
Let $X = \overline{\C} \backslash\{\infty, a_1, \dots, a_m\}$. 
Let $W$ be a finite-dimensional $\L_+$-module, and let $M = A \otimes W$ be the corresponding tensor $AV$-module. Then 
\[
H_\GF^k(V, A \otimes W) \cong H^{k}(\L_+, W) \oplus \left(\C^m \otimes H^{k-1}(\L_+, W) \right).
\]
\end{corollary}


\end{document}